\documentclass[12pt]{article}

\usepackage[utf8]{inputenc}
\usepackage[]{datetime}
\usepackage[noadjust]{cite}
\usepackage{amssymb}
\usepackage{graphicx}
\usepackage{enumerate}
\usepackage{amsfonts}
\usepackage{amsmath}
\usepackage{stmaryrd}
\usepackage{xspace}
\usepackage{hyperref}
\usepackage{amsthm}
\usepackage{xcolor}
\usepackage{tikz}
\usepackage{todonotes}

\hypersetup{
  pdftitle = {Multigraphs without large bonds are wqo by contraction},
  pdfauthor = {M.\ Kaminski, J.-F.\ Raymond and .\ Trunck},
  colorlinks = true,
  linkcolor = black!30!red,
  citecolor = black!30!green
}

\usetikzlibrary{decorations.pathreplacing}
\tikzset{black node/.style={draw, circle, fill = black, minimum size = 5pt, inner sep = 0pt}}
\tikzset{white node/.style={draw, circle, fill = white, minimum size = 5pt, inner sep = 0pt}}
\tikzset{normal/.style = {draw=none, fill = none}}

\newtheorem{theorem}{Theorem}
\newtheorem{lemma}{Lemma}
\newtheorem{corollary}{Corollary}
\newtheorem{proposition}{Proposition}

\theoremstyle{remark}

\theoremstyle{definition}
\newtheorem{definition}{Definition}

\newcommand{\ie}{{i.e.}}
\newcommand{\etal}{{et al.}}

\newcommand{\N}{\mathbb{N}}
\newcommand{\intv}[2]{\left \{ #1,\dots, #2 \right \}}
\DeclareMathOperator{\powset}{\mathcal{P}^{<\omega}}

\newcommand{\lmc}{\unlhd}
\newcommand{\lleq}{\preceq} 

\DeclareMathOperator{\incl}{Incl}

\newcommand{\mref}[1]{\hyperref[#1]{(M\ref*{#1})}}
\newcommand{\itemref}[1]{\hyperref[#1]{(\ref*{#1})}}

\DeclareMathOperator{\lab}{lab} 
\newcommand{\seqb}[1]{\left ( #1 \right )} 
\DeclareMathOperator{\rt}{root} 
\DeclareMathOperator{\mult}{mult} 

\title{Multigraphs without large bonds\\are wqo by~contraction\thanks{This work was partially
    supported by the Warsaw Center of Mathematics and Computer
    Science. Emails:
    \href{mailto:mjk@mimuw.edu.pl}{\texttt{mjk@mimuw.edu.pl}},
    \href{mailto:jean-florent.raymond@mimuw.edu.pl}{\texttt{raymond@tu-berlin.de}},
    and \href{mailto:theophile.trunck@ens-lyon.fr}{\texttt{theophile.trunck@ens-lyon.fr}}.
  }}
\author{Marcin Kamiński\thanks{Institute of Computer Science, University of Warsaw, Poland.}\and Jean-Florent
  Raymond\thanks{Logic and Semantics Research Group, Technische Universität Berlin, Germany. This work has been done while this author was affiliated to the University of Montpellier (LIRMM) and the University of Warsaw.}\and Théophile Trunck\thanks{LIP, ÉNS de Lyon, France.}}
\date{}

\begin{document}
\maketitle

\abstract{
We show that the class of multigraphs with at most $p$ connected components and bonds of size at most $k$ is well-quasi-ordered by edge contraction for all positive integers~$p,k$. (A~\emph{bond} is a minimal non-empty edge cut.) We also characterize canonical antichains for this relation.}

\section{Introduction}
\label{sec:intro}

A \emph{well-quasi-order} (\emph{wqo} for short) is a quasi-order which contains neither infinite decreasing sequences, nor infinite collections of pairwise incomparable elements. The beginnings of the theory of well-quasi-orders go back to the 1950s and some early results on wqos include that of Higman on sequences from a wqo~\cite{Higman01011952}, Kruskal's Tree Theorem \cite{Kruskal60well}, as well as other (now standard) techniques, for example the \emph{minimal bad sequence} argument of Nash-Williams~\cite{Nash63onwe}.

A recent result on wqos and arguably one of the most significant results in this field is  the theorem by Robertson and Seymour which states that graphs are well-quasi-ordered by the minor relation~\cite{Robertson2004325}. Later, the same authors also proved that graphs are well-quasi-ordered by the immersion relation~\cite{Robertson:2010:GMX}.

Nonetheless, most of containment relations do not well-quasi-order the class of all graphs. For example, graphs are not well-quasi-ordered by subgraphs, induced subgraphs, or topological minors. Therefore, attention was naturally brought to classes of graphs where well-quasi-ordering for such relations exists. Damaschke proved that cographs are well-quasi-ordered by induced subgraphs~\cite{JGT:JGT3190140406} and Ding characterized subgraph ideals that are well-quasi-ordered by the subgraph relation~\cite{JGT:JGT3190160509}. Finally, Liu and Thomas recently announced that graphs excluding any graph of a class called ``Robertson chain'' as topological minor are well-quasi-ordered by the topological minor relation~\cite{Lui2014}.

An \emph{antichain} is a collection of non-comparable elements.
Another line of research is to classify non-wqo
containment relations depending on the antichains they contain.
Ding introduced the concepts of \emph{canonical antichain} and \emph{fundamental antichain} aimed at extending the study of the existence of obstructions of being well-quasi-ordered in a quasi-order \cite{Ding20091123}. In particular, he proved that finite graphs do not
admit a canonical antichain under the induced subgraph relation but they do under the subgraph~relation.

In this paper, we consider finite graphs where parallel edges are allowed,
but not loops. Graphs where no edges are parallel are referred to as
\emph{simple graphs}. An~\emph{edge contraction} is the operation that identifies
two adjacent vertices and deletes the loops that were possibly created (but
keeps multiple edges). A graph $H$ is said to be a \emph{contraction} of
a graph $G,$ denoted $H \lmc G,$ if $H$ can be obtained from $G$ by a sequence
of edge contractions. A \emph{bond} is a minimal non-empty edge cut, \ie~a minimal set of edges whose removal increases the number of connected components (cf.~\autoref{fig:bond-house}).
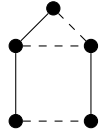
\begin{figure}[h]
  \centering
  \begin{tikzpicture}[every node/.style = black node]
  \begin{scope}
    \node (N0) at (0,0) {};
    \node (N1) at (0,-1) {};
    \node (N2) at (1,-1) {};
    \node (N3) at (1,0) {};
    \node (N4) at (0.5,0.5) {};
    \draw
    (N0) edge (N1)
    (N1) edge[dashed] (N2)
    (N2) edge (N3)
    (N3) edge[dashed] (N4)
    (N4) edge (N0)
    (N0) edge[dashed] (N3);
  \end{scope}
\end{tikzpicture}
  \caption{A bond of order 3 (dashed edges) in the house graph.}
  \label{fig:bond-house}
\end{figure}

The contraction
relation defines a quasi-order on finite graphs. This is not a wqo, as
witnessed by the following infinite antichains, that are also depicted
in \autoref{fig:pkantichain}: $\mathcal{A}_\theta$ is the class of
connected graphs with two vertices and $\mathcal{A}_{\overline{K}}$ is
the class of edgeless graphs with at least one vertex. For the contraction ordering, infinite
antichains are the only obstructions to well-quasi-ordering as
decreasing sequences are always finite.

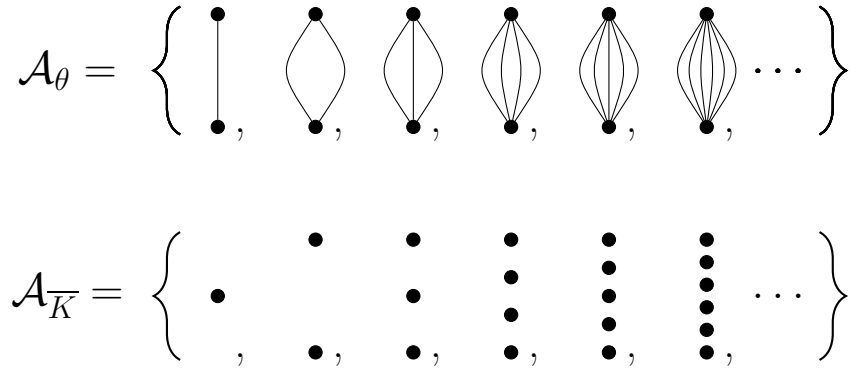
\begin{figure}[h]
  \centering
  \begin{tikzpicture}[every node/.style = black node]
    \def \maxi {5}
    \begin{scope} 
    
      \node[normal] at (-2, 0.75) {{\Large$\mathcal{A}_\theta = $}};
      \node (x) at (0,0) {};
    \node (y) at (0,1.5) {};
    \draw (x) edge (y);
      \node[normal] at (0.3,-0.1) {{\Large,}};
    \foreach \nb in {1,...,\maxi} {
      \begin{scope}[xshift = 1.3*\nb cm]
      \begin{scope}[xscale = 0.5, yscale = 0.75, rotate = 90]
        \node (x) at (0,0) {};
        \node (y) at (2,0) {};
        \foreach \i in {0,...,\nb}{
          \draw (x) .. controls (1, 1 - \i * 2 / \nb) and (1, 1 - \i * 2 / \nb) .. (y);
        }
      \end{scope}
      \node[normal] at (0.3,-0.1) {{\Large,}};
    \end{scope}
    \node[normal] at (7.5, 0.75) {{\Large$\dots$}};
    \draw[thick,decorate,decoration={brace,amplitude=10pt}] (-0.5, -0.1)
    -- (-0.5, 1.6);
    \draw[thick,decorate,decoration={brace,amplitude=10pt}] (8, 1.6)
    -- (8, -0.1);
    }
  \end{scope}
  \begin{scope}[yshift = -2cm] 
    \node[normal] at (-2, -0.25) {{\Large$\mathcal{A}_{\overline{K}} = $}};
    \node (W) at (0,-.25) {};
    \node[normal] at (0.3, -1.1) {{\Large,}};
        \foreach \nb in {1,...,\maxi} {
          \begin{scope}[xshift = 1.3*\nb cm]
          \begin{scope}[scale = 0.5]
            \foreach \i in {0,...,\nb}{
              \node (W\i) at (0, 1 - \i * 3 / \nb) {};
            }
          \end{scope}
          \node[normal] at (0.3, -1.1) {{\Large,}};
        \end{scope}
    }
    \node[normal] at (7.5, -0.25) {{\Large$\dots$}};
    \draw[thick,decorate,decoration={brace,amplitude=10pt}] (-0.5,-1.1)
    -- (-0.5,0.6);
    \draw[thick,decorate,decoration={brace,amplitude=10pt}] (8, 0.6) --
    (8, -1.1);
  \end{scope}

  \end{tikzpicture}
  \caption{Two infinite antichains for contractions: multiedges and cocliques.}
  \label{fig:pkantichain}
\end{figure}

For every $p,k \in \N,$ let $\mathcal{G}_{p,k}$ be the class of
graphs having at most $p$ connected components and not containing a
bond of order more than $k.$ Our main result is the~following.

\begin{theorem}\label{t:main}
  For every $p,k \in \N,$ the class $\mathcal{G}_{p,k}$ is well-quasi-ordered by $\lmc.$
\end{theorem}

The complement of a simple graph $G$, denoted $\overline{G}$ is the
graph obtained by replacing every edge by a non-edge and vice-versa
in~$G$. Remark that a graph has a bond of order $k$ iff it contains
$\theta_k$ as contraction, and that it has 
$p$ connected components iff it can be contracted to~$\overline{K}_p.$
A class
$\mathcal{G}$ of graphs is said to be \emph{contraction-closed} if $H
\in \mathcal{G}$ whenever $H \lmc G$ for some~$G \in \mathcal{G}.$ 
As a consequence of our main theorem and of the fact that each
infinite subset of
$\mathcal{A}_\theta$ or $\mathcal{A}_{\overline{K}}$ is an
obstruction to be well-quasi-ordered, we have the following characterization.

\begin{corollary}\label{c:ccwqo}
  A contraction-closed class $\mathcal{H}$ is well-quasi-ordered by
  $\lmc$ iff there are $k,p \in \N$ such that $\forall k'>k,\
  \theta_{k'} \not \in \mathcal{H}$ and $\forall p'>p,\ \overline{K}_p
  \not \in \mathcal{H}.$
\end{corollary}

In his study of infinite antichains for the (induced) subgraph
relation, Ding~\cite{Ding20091123} introduced the two following concepts.
An antichain $\mathcal{A}$ of a quasi-order $(\mathcal{S}, \lleq)$
is said to be \emph{canonical} if the following holds for every
$\lleq$-closed subclass $\mathcal{J}$ of $\mathcal{S}$: $\mathcal{J}$ has an
infinite antichain iff $\mathcal{J} \cap \mathcal{A}$ is
infinite. Canonical antichains can be used to characterize the $\lleq$-closed
subclasses of a quasi-order $(\mathcal{S}, \lleq)$ and also to
describe the variety of its antichains.

The following result is a complete characterization of the
canonical antichains of $\lmc$ in finite graphs, which extends the
results of Ding on canonical antichains of simple graphs for the
relations of subgraph and induced subgraph~\cite{Ding20091123}.

\begin{theorem}\label{t:can-anti}
  An antichain $\mathcal{A}$ of $\lmc$ is canonical
  iff $\mathcal{A}_{\theta} \cup \mathcal{A}_{\overline{K}}\setminus \mathcal{A}$ is finite.
\end{theorem}

In other words, an antichain $\mathcal{A}$ is canonical iff it contains all but
finitely many graphs from $\mathcal{A}_\theta$ and $\mathcal{A}_{\overline{K}}$. As noted in
the proof, this implies that a canonical antichain contains a finite number of graphs that do not belong to~$\mathcal{A}_\theta \cup \mathcal{A}_{\overline{K}}$. In the course of studying antichains, we also obtained general results that are of independent interest. They are presented in the corresponding section.

\paragraph{Organization of the paper.}
The notions and tools that are used in this paper are introduced in
\autoref{s:prelim}. \autoref{t:main} is proved in
\autoref{sec:newproof} and results pertaining to antichains appear in
\autoref{sec:ca}.

\paragraph{Conclusion.}
This work settles the case of multigraph contractions in the study of
well-quasi-ordered subclasses, a problem investigated by
Damaschke for induced subgraphs~\cite{JGT:JGT3190140406}, by Ding
for subgraphs~\cite{JGT:JGT3190160509} and induced
minors~\cite{JGT:JGT4}, and by Fellows \etal~for
several containment relations~\cite{fellows2009well}.
In particular, we give necessary and sufficient conditions for a class
of (multi)graphs to be well-quasi-ordered by multigraph
contractions. Furthermore, we characterize canonical antichains for
this relation, in the continuation
of  Ding's results for subgraph and contraction relation in~\cite{Ding20091123}.

\section{Preliminaries}
\label{s:prelim}

We denote by $V(G)$ the set of vertices of a graph $G$
and by $E(G)$ its multiset of edges. Given two adjacent vertices
$u,v$ of a graph $G,$ $\mult_G(\{u,v\})$ stands for the number of
parallel edges between~$u$ and~$v,$ called \emph{multiplicity} of the
edge $\{u,v\}.$
We denote by $\powset(S)$ the class of finite subsets of a set $S$ and by
$\mathcal{P}(S)$ its power set.

\begin{definition}
  A \emph{model} of $H$ in $G$ (\emph{$H$-model} for short) is a function $\mu \colon V(H) \to
  \mathcal{P}(V(G))$ such~that:
  \begin{enumerate}[(M1)]
  \item $\mu(u)$ and $\mu(v)$ are disjoint whenever $u,v \in
    V(H)$ are distinct;\label{e:m1}
  \item $\{\mu(u)\}_{u \in V(H)}$ is a partition of
    $V(G)$;\label{e:m2}
  \item for every $u \in V(H)$, the graph $G[\mu(u)]$
    is connected;\label{e:m25}
  \item for every $u,v \in V(H),$ $\mult_H(\{u,v\}) = \sum_{(u',v')
      \in \mu(u) \times \mu(v)} \mult_G(\{u',v'\})$.\label{e:m3}
  \end{enumerate}  
\end{definition}

Remark that $H$ is a contraction of $G$ iff $G$ has an
$H$-model. Intuitively, an $H$-models indicates which connected
subgraphs to contract in $G$ in order to obtain a graph isomorphic to~$H$.

\paragraph{Quasi-orders.}
A \emph{quasi-ordered set} (\emph{qoset} for short) is a pair $(\Sigma,
\lleq)$ where $\Sigma$ is a set and $\lleq$ is a binary relation on $S$
which is reflexive and transitive.
An \emph{antichain} is a sequence of pairwise non-comparable elements.
We say that $(\Sigma, \lleq)$ is a \emph{well-quasi-order}
(\emph{wqo} for short), or that its elements are \emph{well-quasi-ordered} by $\lleq,$ if it has neither an infinite decreasing sequence, nor an infinite antichain.

\begin{definition}[Cartesian production of qosets]
If $(\Sigma, \lleq_\Sigma)$ and $(\Sigma', \lleq_{\Sigma'})$ are two qosets, then
their \emph{Cartesian product} $(\Sigma \times \Sigma', \lleq_\Sigma \times \lleq_{\Sigma'})$
  is the qoset defined~by:
  \[
  \forall (x,x'),(y,y') \in \Sigma \times \Sigma',\  (x,x') \lleq_\Sigma \times \lleq_{\Sigma'}
  (y,y')\ \text{if}\ x \lleq_\Sigma y\ \text{and}\ x' \lleq_{\Sigma'} y'.
  \]
\end{definition}

\begin{proposition}[Higman~\cite{Higman01011952}]\label{p:wqo-product}
  If $(\Sigma, \lleq_\Sigma)$ and $(\Sigma', \lleq_{\Sigma'})$ are wqo, then so
  is~$(\Sigma \times \Sigma', \lleq_{\Sigma} \times \lleq_{\Sigma'}).$
\end{proposition}

\paragraph{Sequences.}
We~use the notation $\Sigma^\star$ for the class of all finite
sequences over the set $\Sigma$ (including the empty sequence).
For any qoset $(\Sigma, \lleq),$ we define
the relation $\lleq^\star$ on $\Sigma^\star$ as follows: for every $r
=\seqb{r_1,\dots, r_p}$ and $s = \seqb{s_1,\dots, s_q}$ of
$\Sigma^\star,$ we have $r \lleq^\star s$ if there is a increasing
function $\varphi \colon \intv{1}{p} \to \intv{1}{q}$ such that for
every $i \in \intv{1}{p}$ we have~$r_i \lleq s_{\varphi(i)}.$ This
generalizes the subsequence relation.
This order relation is extended to the class $\powset(\Sigma)$ of finite subsets of $\Sigma$ as follows,
generalizing the subset relation: for every $A,B \in \powset(\Sigma)$, we
write $A \lleq^\star B$ if there is an injection $\varphi \colon A \to
B$ such that $\forall x \in A,\ x \lleq \varphi(x).$

\begin{proposition}[Higman~\cite{Higman01011952}]\label{p:highman-wqolabel}
  If $(\Sigma, \lleq)$ is a wqo, then so is $(\Sigma^\star, \lleq^\star).$
\end{proposition}

Labeled graphs as defined below will allow us to focus on
2-connected graphs.
\begin{definition}[labeled graph]
  Let $(\Sigma, \lleq)$ be a qoset. A $(\Sigma, \lleq)$-labeled graph is a pair $(G,
  \lambda)$ where $\lambda \colon V(G) \to
  \powset(\Sigma)$ is a function, referred to as the \emph{labeling of the
    graph}. If $\mathcal{H}$ is a class of (unlabeled) graphs, 
  $\lab_{(\Sigma, \lleq)}(\mathcal{H})$ denotes the class of $(\Sigma, \lleq)$-labeled graphs of
  $\mathcal{H}.$
\end{definition}
For simplicity, we will use the same symbol $G$ to refer to a labeled
graph and its underlying unlabeled graph and we will denote by $\lambda_G$ its labeling function.
Remark that any unlabeled graph can be seen as a
$\emptyset$-labeled~graph.

The contraction relation is extended to labeled graphs by additionally
allowing to relabel by $l'$ any vertex labeled $l$ whenever $l' \lleq
l$. In terms of models, this corresponds to the following extra requirement for $\mu$ to
be a model of $H$ in $G$:
\begin{enumerate}[(M1)]\setcounter{enumi}{4}
\item $\forall v \in V(H),\ \lambda_H(v)
\lleq^\star \bigcup_{u' \in \mu(u)}\lambda_G(u')$.\label{e:m5}
\end{enumerate}

\section{Small bonds and well-quasi-ordering}
\label{sec:newproof}

We first show that we can focus on (labeled) 2-connected graphs.

\begin{lemma}\label{l:higmgraph}
  Let $p$ be a positive integer, let $\mathcal{H}$ be a class of connected graphs, and let $\mathcal{H}'$ be the class of graphs with at most $p$ connected components, each of which belongs to $\mathcal{H}$. If $(\mathcal{H}, \lmc)$ is a wqo, then so is $(\mathcal{H}', \lmc)$.
\end{lemma}

\begin{proof}
  Let us observe that for every $q\in \N$, the class $\mathcal{H}_q$ of graphs with exactly $q$ connected components, each of which belongs to $\mathcal{H}$, is wqo by~$\lmc$. This follows from Higman's Lemma (\autoref{p:wqo-product}) as these graphs can be seen as sequences of exactly $q$ graphs of $\mathcal{H}$.
  As a finite union of wqos, $(\mathcal{H}', \lmc)$ is a wqo. Indeed, if $\mathcal{H}'$ had an infinite antichain, this antichain would have infinite intersection with $\mathcal{H}_i$ for some $i\in \intv{0}{p}$, a contradiction since $(\mathcal{H}_i, \lmc)$ is a wqo.
\end{proof}

\begin{lemma}\label{l:2clab-wqo}
  Let $\mathcal{H}$ be a contraction-closed class of connected
  graphs and let $\mathcal{H}^{(2)}$ be the subclass of 2-connected
  graphs of~$\mathcal{H}.$ If for every wqo $(\Sigma, \lleq)$, the qoset
  $(\lab_{(\Sigma, \lleq)}(\mathcal{H}^{(2)}), \lmc)$ is a wqo, then
  so is~$(\mathcal{H}, \lmc)$.
\end{lemma}

\begin{proof}
  This proof is very similar to the induced minor case proved
  in~\cite{fellows2009well}.
  We deal here with rooted graphs, that are graphs with a
  distinguished vertex called \emph{root}. We denote the
  root of a rooted graph~$G$ by $\rt(G)$.
  The contraction relation is extended to this setting by requiring
  roots to be contracted to roots. Formally, $\rt(G) \in
  \mu(\rt(H))$ for every model of $H$ in $G$.
  Assuming that $(\mathcal{H}, \lmc)$ is not a wqo,
  we consider the class $\mathcal{H}_r$ of all graphs that can be
  obtained by choosing a root in a graph of $\mathcal{H}$. Clearly,
  this class is not well-quasi-ordered by $\lmc$.

  In a sequence of graphs, $(H,G)$ is a \emph{good pair} if $H$ appear before
  $G$ and $H \lmc G$.
  We use the concept of \emph{bad sequence},
  that are infinite sequences with no good pair. The absence of bad
  sequences in a class of graphs is equivalent to this class being well-quasi-ordered (see \cite{Nash63onwe}).
  Towards a contradiction, we suppose the existence of a bad
  sequence. We consider a minimal one, in the following sense.
  Let $\seqb{G_i}_{i\in \N}$ be a bad sequence of graphs of $\mathcal{H}_r$ such
  that for every $i \in \N,$ there is no contraction $G$ of $G_i$
  (distinct from $G_i$) such that a bad
  sequence starts with $G_0, \dots, G_{i-1}, G$.

  For every $i\in \N$, let $A_i$ be the
  maximal 2-connected subgraph of $G_i$ which contains $\rt(G_i)$. Let $C_i$ the set of cutvertices of $G_i$ that belong to $A_i$. For each cutvertex $c \in C_i$, let $B_c^i$ be the connected
  component of $G_i\setminus (V(A_i)\setminus C_i)$ turned into a rooted graph
  by setting $\rt(B_c^i)=c$. Note that we have $B_c^i \lmc G_i$.
  Let us denote by $\mathcal{B}$ the family of rooted graphs $\mathcal{B} =
  \{B_c^i \colon c\in C_i, i\in \N\}$. We will show that $(\mathcal{B}, \lmc)$ is a
  wqo. Let $\seqb{H_j}_{j\in \N}$ be an infinite sequence in $\mathcal{B}$ and for every
  $j\in \N$ choose an $i = \varphi(j) \in \N$ for which $H_j \lmc G_i$. Pick
  a $j$ with smallest $\varphi(j)$, and consider the sequence $G_1, \dots,
  G_{\varphi(j)-1}, H_j, H_{j+1}, \dots$. By minimality of $\seqb{G_i}_{i\in \N}$
  and by our choice of $j$, since $H_j \lmc G_{\varphi(j)}$ and $H_j \neq G_{\varphi(j)}$,
  this is not a bad sequence so it contains a good pair $(G, G')$. Now, if $G$ is
  among the first $\varphi(j)-1$ elements, then as $\seqb{G_i}_{i\in \N}$ is bad we
  must have $G' = H_{j'}$ for some $j'\geq j$ and so we have $G_{i'} = G\lmc
  G' = H_{j'} \lmc G_{\varphi(j')}$, a contradiction. So there is a
  good pair in $\seqb{H_i}_{i \geq j}$ and hence the infinite sequence
  $\seqb{H_j}_{j\in \N}$ has a good pair, so $(\mathcal{B}, \lmc)$ is
  a~wqo.

  We will now find a good pair in $\seqb{G_i}_{i\in \N}$ to show a
  contradiction. The idea is to label the graph family $\mathcal{A}=\{A_i\}_{i
    \in \N}$ so that each cutvertex $c$ of a graph $A_i$ gets labeled by their
  corresponding connected component $B_c^i$, and the roots are preserved under
  this labeling. More precisely, for each $A_i$ we define a labeling
  $\sigma_i$ that assigns to every vertex $v \in V(G_i)$ a label
  $\{(\sigma_i^1(v), \sigma_i^2(v))\}$ defined as follows:

  \begin{itemize}
    \item $\sigma_i^1(v) = 1$ if $v = \rt(G_i)$ and $\sigma_i^1(v) = 0$~otherwise;
    \item $\sigma_i^2(v) = B_v^i$ if $v\in C_i$ and $\sigma_i^2(v)$ is
      the one-vertex rooted graph~otherwise.
  \end{itemize}

  The labeling $\sigma$ of $\mathcal{A}$ is then $\{\sigma_i : i\in
  \N\}$. Let us define a quasi-ordering $\lleq$ on the set of labels
  $\Sigma$ assigned by $\sigma$. For two labels $(s_a^1, s_a^2),(s_b^1, s_b^2)\in \Sigma$ we
  define $(s_a^1, s_a^2) \lleq (s_b^1, s_b^2)$ iff $s_a^1 = s_b^1$
  and $s_a^2 \lmc s_b^2$. Note that in this situation, $s_a^2$ and
  $s_b^2$ are rooted graphs, so $\lmc$ compares rooted graphs. Observe
  that since $(\mathcal{B},\lmc)$ is wqo, then $(\Sigma, \lleq)$ is
  wqo. For every $i \in \N$, let $A'_i$ be the  $(\Sigma,
  \lleq)$-labeled rooted graph $(A_i, \sigma_i)$. We now consider the
  infinite sequence $\seqb{A'_i}_{i \in \N}$. By our initial
  assumption, $(\lab_\Sigma(\mathcal{A}), \lmc)$ is wqo (as
  $\mathcal{A}$ consists only in 2-connected graphs), so there is a
  good pair $(A'_i, A'_j)$ in the sequence~$\seqb{A'_i}_{i \in \N}$.

  To complete the proof, we will show that $A'_i\lmc A'_j \Rightarrow G_i\lmc G_j$.
  Let $\mu$ be a model of $A'_i$ in $A'_j$. Then for each
  cutvertex $c\in C_i$, $\mu(c)$ contains a vertex $d \in C_j$ with $B_c^i\lmc
  B_d^j$. Let $\mu_c$ denote a model of
  $B_c^i$ onto $B_d^i$. We construct a model $g$ as follows:

\[
\nu \colon \left \{
  \begin{array}{rcl}
    V(G_i) &\to& \mathcal{P}(V(G_j))\\
    v &\mapsto& \mu(v) \text{ if } v\in A_i\setminus C_i\\
    v &\mapsto& \mu_c(v) \text{ if } v\in B_c^i\setminus C_i\\
    v &\mapsto& \mu(v) \cup \mu_v(v) \text{ if } v \in C_i
  \end{array}
\right .
\] 

We now prove that $\nu$ is a model of $G_i$ onto $G_j$. First
note that by definition of $\mu$ and each $\mu_c$, we have $\nu(u)\cap \nu(v) =
\emptyset$ for any pair of distinct vertices $u$ and $v$ in
$G_i$, and also every vertex of $G_j$ is in the image of some vertex
of $G_i$ (items \mref{e:m1} and \mref{e:m2} in the definition of a model). If $u\in C_i$, then $\mu(u)$ contains a vertex
$v\in C_j$ for which $B_u^i \lmc B_v^j$, and $v$ is also contained in
$\mu_v(v)$ since $\mu_v$ preserves roots. Thus, 
$G_j[\nu(u)]$ is connected when $u\in C_i$ (item
\mref{e:m25}). This is obviously true when $u \not \in C_i$ again by
the definitions of $\mu$ and each $\mu_c$. Moreover, the endpoints of
every edge of $G_i$ belong either both to $A_i$, or both to $B^i_c$,
so item \mref{e:m3} follows from the properties of $\mu$ and each
$\mu_c$. Finally, as the labeling $\sigma$ ensures that $\rt(G_j)\in
\nu(\rt(G_i))$, we establish that $G_i\lmc G_j$. So $\seqb{G_i}_{i\in
  \N}$ has a good pair $(G_i, G_j)$, a contradiction.
\end{proof}

Then we show that removing a maximum bond in a 2-connected graphs yields a graphs with smaller maximum bonds. This will be used in the inductive proof of \autoref{t:main}.
\begin{lemma}\label{l:bond}
  If $B \subseteq E(G)$ is a bond of maximum order in a 2-connected graph $G$, then none of
  the components of $G \setminus B$ has a bond of order~$|B|$.
\end{lemma}

\begin{proof}
By contradiction, let us assume that one component of $G \setminus  B$
has a bond $C$ with $|C| = |B|$.
Let $\{X,Y,Z\}$ be a partition of $V(G)$ be such that $G[X]$ and $G[Y\cup
Z]$ are the connected components of $G \setminus B$ and $G[Y]$ and
$G[Z]$ are those of $G[Y \cup Z] \setminus
C$, while no edge of $C$ is incident with a vertex of~$X$.

\noindent \textit{First case:} $B$ is incident with vertices of both
$Y$ and $Z$. Let $B'$ be the set of the edges of $B$ that are incident with
vertices of~$Y$.
Then observe that $G[X \cup Z]$ is connected and is a
connected component of $G \setminus (B' \cup C)$, the other one being $G[Y]$.
Therefore, $B' \cup C$ is a cut and $|B' \cup C| > |B|$, which
contradicts the maximality of~$B$. Therefore this case is not possible.

\noindent \textit{Second case:} $B$ is incident to vertices of exactly
one of $Y$ and $Z$. Without loss of generality, let us assume that
this set is $Y$.
Let $x \in X$ and $y \in Z$. As $G$ is 2-connected, there are two
paths $Q_1$ and $Q_2$ connecting $x$ to $y$ and sharing only
their endpoints. Let $Q_1'$ and $Q_2'$ be minimal subpaths of $Q_1$ and $Q_2$,
respectively, containing exactly one edge from $B$ and $C$. Let $D$ be a
minimum cut of $G[Y]$ separating the internal vertices of $Q_1$ from
that of $Q_2$. As $G[Y]$ is connected, $D \neq \emptyset$.
For every $i \in \{1,2\}$, let $Y_i$ be the connected component of
$G[Y] \setminus D$ containing the internal vertices of $Q_i'$ and let
$B_i$ (resp.\ $C_i$) be the edges of $B$ (resp.\ $C$) that have an
endpoint in~$Y_i$. As $\{B_1,B_2\}$ is a partition of $B$, there is an $i\in \{1,2\}$ such that $|B_i| \geq |B|/2$. Similarly, there
is a $j \in \{1,2\}$ such that $|C_j| \geq |C|/2 = |B|/2$. Let $B' =
B_i \cup C_j \cup D$ and let us show that it is a bond.
As $|B'|>|B|$, this would provide the desired contradiction to the
maximality of $B$.
If $i = j$,
then $B'$ separates $Y_i$ from $X \cup Y_{3-i} \cup Z$. The subgraph
$G[Y_i]$ is connected by definition of $Y_i$, and $G[X \cup Y_{3-i}
\cup Z]$ is because of the path~$Q_{3-i}'$.
If $i \neq j$, then $B'$ separates $Y_i \cup Z$ from $X\cup
Y_{3-i}$. There vertex sets induce connected subgraphs thanks to the paths
$Q'_i$ and $Q'_{3-i}$, respectively. Therefore $B'$ is a
bond and we are done.
\end{proof}

We are now ready to prove \autoref{t:main}.
\begin{proof}[Proof of \autoref{t:main}]
  Our goal is to show that for every $p,k \in \N$, the class $\mathcal{G}_{p,k}$ is well-quasi-ordered by~$\lmc$.
  According to \autoref{l:higmgraph}, if $\left (\mathcal{G}_{1,k}, \lmc \right)$ is a wqo, then so is $\left (\mathcal{G}_{p,k}, \lmc\right )$.
  Therefore we only need to consider the case where~$p=1$ (the case $p=0$ being trivial).
  Observe that one cannot increase the maximum order of a bond in a graph nor disconnect it by performing contractions. Hence $\mathcal{G}_{1,k}$ is contraction-closed. Therefore we can apply \autoref{l:2clab-wqo}, which allows us to focus on labeled 2-connected graphs. We call $\mathcal{G}_{1,k}'$ the class of graphs of $\mathcal{G}_{1,k}$ that are 2-connected.

  The proof then goes by induction on $k$.
  When $k=0$, then $\mathcal{G}'_{1,k}$ is  empty, so $\lab_{(\Sigma, \lleq)}\left (\mathcal{G}'_{1, k} \right )$ is trivially well-quasi-ordered, for every wqo $(\Sigma, \lleq)$.
  Let us now assume that $k>0$ and that for and every wqo $(\Sigma,
  \lleq)$, the class $\lab_{(\Sigma, \lleq)}\left (\mathcal{G}'_{1,
      k-1} \right )$ is well-quasi-ordered by $\lmc$ (induction
  hypothesis). By the remarks above, $\left (\lab_{(\Sigma,
      \lleq)}\left (\mathcal{G}_{p,k-1}\right ), \lmc \right)$ is a
  wqo, for every $p \in \N$.

  Let $(\Sigma, \lleq)$ be a wqo. By contradiction, we assume that $\lab_{(\Sigma, \lleq)}\left
    (\mathcal{G}_{1,k}\right )$ is not a well-quasi-order.
  Let $\{G_i\}_{i \in \N}$ be an infinite antichain in $\lab_{(\Sigma,
    \lleq)}\left (\mathcal{G}_{1,k}\right )$ such that, for some
  $k'\leq k$ and for every $i\in \N$, a largest bond $B_i$ in $G_i$
  has order $k'$.
  As no graphs of $\mathcal{G}_{1,k}$ has a bond of order more than
  $k$, such antichain always exist.

  We may also assume that,
if we respectively denote by $\{x_i^1, \dots, x_i^{k'}\}$ and
$\{y_i^1, \dots, y_i^{k'}\}$ the endpoints of the edges of $B_i$ in
the two connected components of $G_i \setminus B_i$, the bipartite
graph between $\{x_i^1, \dots, x_i^{k'}\}$ and $\{y_i^1, \dots,
y_i^{k'}\}$ is the same for every $i$.
We mean here that $\mult_{G_i}(\{x_i^l, y_i^{l'}\}) =
\mult_{G_j}(\{x_j^l, y_j^{l'}\})$ for every $l, l' \in \intv{1}{k'}$.
This is possible as there is a finite number of bipartite graphs on
$k$ edges and at most $2k$ vertices.

For every $i\in \N$, let $G_i' = G_i \setminus B_i$.
Let $H_i$ be a copy of $G_i'$ (with the same vertex set) that we label as follows.
Let $v \in V(H_i)$. If $v = x_i^{j}$ (resp.\ $v = y_i^{j}$) for some $j \in \intv{1}{k'}$,
then we set $\lambda_{H_i}(v) = (\lambda_{G_i'}(v), j)$ (resp.\ $\lambda_{H_i}(v) = (\lambda_{G_i'}(v), 2j)$), otherwise we set $\lambda_{H_i}(v) = (\lambda_{G_i'}(v), 0)$.

Let $\Sigma'  = \Sigma \times \intv{0}{2k-1}$, let $\lleq'$ be the Cartesian product of $\lleq$ and $=$ and notice that, as a Cartesian product of wqos, $(\Sigma', \lleq')$ is a wqo (\autoref{p:wqo-product}).
According to \autoref{l:bond}, the graph $H_i$ belong to $\lab_{(\Sigma', \lleq')}\left (\mathcal{G}_{2,k-1}\right )$.
By induction hypothesis, $\{H_i\}_{i \in \N}$ is wqo by~$\lmc$.
Let $i$ and $j$ be distinct integers such that $H_i \lmc H_j$ and let $\mu$ be a model of $H_i$ in~$H_j$.
Let us show that $\mu$ is also a model of $G_i$ in $G_j$. The
properties \mref{e:m1}, \mref{e:m2}, and \mref{e:m25} follow
from the definition of~$\mu$, as well as \mref{e:m3} when $u$ and
$v$ belong to the same connected component of~$H_i$. Also, as $V(G_i)
= V(H_i)$, $\mu$ seen as model of $G_i$ in $G_j$ satisfies \mref{e:m5}. Therefore we only need to prove that \mref{e:m3} holds when $u$ and $v$ belong to distinct connected components of $H_i$.

Let $l \in \intv{1}{k'}$. Since $\lambda_{G_i}(x_i^l)$ is only
comparable with the labels that have $l$ as their second coordinate, and
that $x_j^l$ is the only vertex of $G_j$ with that property, we get
$x_j^l \in \mu(x_i^l)$. Similarly, $y_j^l \in \mu(y_i^l)$. The
image of $\mu$ consists of disjoint subsets, thus we have that if $v\in
V(G_i)$ is not of the form $x_i^l$ or $y_i^l$ for some $l \in
\intv{1}{k'}$, then $\mu(v) \cap \bigcup_{l=1}^{k'}\{x_j^l, y_j^l\} =
\emptyset$. As the possible edges in $G_j$ between two vertices $u,v$
that belong to distinct connected components of $H_i$ are the edges of
$B_i$, we deduce that if one of $u,v$ is not of the form $x_i^l$ or
$y_i^l$ for some $l \in \intv{1}{k'}$, then there is no edge between a
vertex of $\mu(u)$ and one of~$\mu(v)$. This proves \mref{e:m3} in
this case. The case where $u = x_i^l$ and $v = y_j^{l'}$ for some
$l,l' \in \intv{1}{k'}$ follows from our choice of the antichain
$\{G_i\}_{i\in \N}$ with the property that $\mult_{G_i}(\{x_i^l, y_i^{l'}\}) =
\mult_{G_j}(\{x_j^l, y_j^{l'}\})$.
\end{proof}

\section{On canonical antichains}
\label{sec:ca}

We present in this section general lemmas on antichains and prove \autoref{t:can-anti}.
If $S$ is the subset of a qoset $(\Sigma, \lleq)$, the $\lleq$-\emph{closure} of $S$  is defined as
$\{x,\ x \lleq y\ \text{for some}\ y \in S\}$.
For every $A \subseteq \Sigma$, we define $\incl_\lleq(A) = \{x \in \Sigma,\ x \prec
a\ \text{for some}\ a \in A\}$. An antichain $A$ of $(\Sigma, \lleq)$ is said to be a \emph{fundamental antichain} if $\incl_{\lleq}(A)$ has no infinite antichains. This concept was introduced by \cite{Ding20091123} in its study of canonical antichains.

\begin{lemma}\label{l:coucou}
  Let $(\Sigma, \lleq)$ be a qoset and let $A$ be an antichain. If $A$ is canonical, then $A$ is fundamental.
\end{lemma}

\begin{proof}
  Let $C = \incl_{\lleq}(A)$.
  As $C$ is $\lleq$-closed, it contains infinite antichains iff $C \cap A$ is infinite.
  However, for every $x \in C$ there is a element $y \in A$ such that $x \prec y$.
  Therefore, if $x \in C \cap A$, then $A$ contains two distinct elements that are comparable. We deduce $C \cap A = \emptyset$. Therefore $C$ has no infinite antichains: $A$ is canonical.
\end{proof}

A consequence of \autoref{l:coucou} and \autoref{t:can-anti} is that every antichain $\mathcal{A}$ of $\lmc$ such that $\mathcal{A}_{\theta} \cup \mathcal{A}_{\overline{K}}\setminus \mathcal{A}$ is finite is fundamental. In the following lemma, we formalize and extend observations on canonical antichains made in \cite{Ding20091123}.

\begin{lemma}\label{l:cancan}
  Let $(\Sigma, \lleq)$ be a qoset and let $A,B$ be two infinite antichains:
  \begin{enumerate}[(i)]
  \item if $A$ and $B$ are canonical, $A\setminus B$ and $B \setminus A$ are finite;\label{ii1}
  \item if $A$ is canonical and $A\setminus B$ is finite then $B$ is
  canonical and $B \setminus A$ is finite.\label{ii2}
  \end{enumerate}
\end{lemma}

\begin{proof}
\textit{Proof of \itemref{ii1}.}
  We consider $D = (B\setminus A) \cup
  \incl(B \setminus A)$, the $\lleq$-closure of $B\setminus A$.
  If $B \setminus A$ is infinite, then $D$ contains an infinite
  antichain. Therefore, as $A$ is canonical, $D \cap A$
  is infinite. However, $A \cap \incl(B)$ is finite because $B$, being
  canonical, is fundamental (\autoref{l:coucou}).
  Hence $A \cap (B\setminus A)$ is infinite, a contradiction.

\noindent \textit{Proof of \itemref{ii2}.}
  By definition, for every $\lleq$-closed class $C$, $C$ has an infinite antichain iff $C \cap A$ is infinite.
  As $A \setminus B$ is finite, $C \cap A$ is infinite iff $C \cap A \cap B$ is infinite. This implies that $C\cap B$ is infinite.
  On the other hand, if $C \cap B$ is infinite, then $C$ clearly has
  an infinite antichain. Therefore $B$ is a canonical antichain. Applying \itemref{ii1} we get that $B \setminus A$ is finite.
\end{proof}

\autoref{t:can-anti} is a consequence of \autoref{l:cancan} applied with $\mathcal{A}$ and $\mathcal{A}_\theta \cup \mathcal{A}_{\overline{K}}$ which is, as a consequence of \autoref{c:ccwqo}, a canonical antichain.

\section*{Acknowledgements}
\addcontentsline{toc}{section}{Acknowledgements}
The authors thank the anonymous referee whose comments allowed us to substantially shorten the paper.


\end{document}